\documentclass[12pt]{article}
\usepackage[cm]{fullpage}

\usepackage{amsmath,amscd, float,times,rotating}
\usepackage{pb-diagram}
\usepackage{color}

\usepackage{latexsym}
\usepackage{amsfonts}
\usepackage{amsmath}
\usepackage{amssymb}

\numberwithin{equation}{section}
\newcommand{\qed}{\hfill \ensuremath{\Box}}

\def\XXint#1#2#3{{\setbox0=\hbox{$#1{#2#3}{\int}$}
\vcenter{\hbox{$#2#3$}}\kern-.5\wd0}}

\newcommand{\tr}{\textnormal{tr}}

\newcommand{\ga}{\gamma}
\newcommand{\de}{\delta}

\newcommand{\ve}{\varepsilon}

\newcommand{\dbar}{\overline{\partial}}

\newcommand{\ddt}[1]{\frac{\partial #1}{\partial t}}

\newcommand{\ov}[1]{\overline{#1}}

\newcommand{\ofs}{\omega_{\textrm{FS}}}

\newcommand{\ddbar}{\frac{\sqrt{-1}}{2\pi} \partial\dbar}

\begin{document}
\newcounter{remark}
\newcounter{theor}
\setcounter{remark}{0} \setcounter{theor}{1}
\newtheorem{claim}{Claim}
\newtheorem{theorem}{Theorem}[section]
\newtheorem{proposition}{Proposition}[section]
\newtheorem{lemma}{Lemma}[section]
\newtheorem{definition}{Definition}[section]
\newtheorem{corollary}{Corollary}[section]
\newenvironment{proof}[1][Proof]{\begin{trivlist}
\item[\hskip \labelsep {\bfseries #1}]}{\end{trivlist}}
\newenvironment{remark}[1][Remark]{\addtocounter{remark}{1} \begin{trivlist}
\item[\hskip \labelsep {\bfseries #1
\thesection.\theremark}]}{\end{trivlist}}

\begin{center}
{\Large \bf The K\"ahler-Ricci flow on Hirzebruch surfaces
\footnote{Research supported in part by National Science Foundation
grants  DMS-06-04805 and DMS-08-48193.  The second-named author is also supported in part by a Sloan Research Fellowship.
 }}
\\
\bigskip
\bigskip

{\large Jian Song$^{*}$ and
Ben Weinkove$^\dagger$} \\

\end{center}

\bigskip

\bigskip
\bigskip
\noindent
{\bf Abstract} \ We investigate the metric behavior of the K\"ahler-Ricci flow on the Hirzebruch surfaces, assuming the initial metric is invariant under a maximal compact subgroup   of the automorphism group.  We show that, in the sense of Gromov-Hausdorff, the flow either  shrinks to a point, collapses to $\mathbb{P}^1$ or contracts an exceptional divisor,  confirming a conjecture of Feldman-Ilmanen-Knopf.  We also show that similar behavior holds on  higher-dimensional analogues of the  Hirzebruch surfaces.


\section{Introduction}

The behavior of the K\"ahler-Ricci flow on a compact manifold $M$ is expected to reveal the metric and algebraic structures on $M$.  If $M$ is a K\"ahler manifold with $c_1(M)=0$ then the K\"ahler-Ricci flow $\partial \omega/\partial t = - \textrm{Ric}(\omega)$ starting at a metric   $\omega_0$ in any K\"ahler class $\alpha$, converges to the unique Ricci-flat metric in $\alpha$ \cite{Cao1}, \cite{Y1}.  If  $c_1(M) <0$, the normalized K\"ahler-Ricci flow
\begin{eqnarray} \label{KRF1}
\ddt{} \omega  =  \mbox{} \lambda \omega -  \textrm{Ric}(\omega), \quad \omega(0)  =  \omega_0,
\end{eqnarray}
with $\lambda=-1$ and $\omega_0 \in \lambda c_1(M)$ converges to the unique K\"ahler-Einstein metric \cite{Cao1}, \cite{Y1}, \cite{A}.

If $c_1(M)>0$ then K\"ahler-Einstein metrics do not exist in general.  If one assumes the existence of  a K\"ahler-Einstein metric then, according to unpublished work of Perelman \cite{P2} (see \cite{TZhu}), the normalized K\"ahler-Ricci flow (\ref{KRF1}) with $\lambda=1$ and $\omega_0 \in \lambda c_1(M)$  converges to a K\"ahler-Einstein metric (this is due to \cite{H}, \cite{Ch} in the case of one complex dimension).
By a conjecture of Yau \cite{Y2},  a necessary and sufficient condition for $M$ to admit a K\"ahler-Einstein metric  is that $M$ be  `stable in the sense of geometric invariant theory'. Tian \cite{T1} later proposed the condition  of \emph{K-stability}  and this concept has been refined and extended by Donaldson \cite{D}.
One might expect that the sufficiency part of the Yau-Tian-Donaldson conjecture can be proved via the flow (\ref{KRF1}).  Indeed, the problem of using stability conditions to prove convergence properties of the K\"ahler-Ricci flow is
 an area of considerable current interest and we refer the reader to   \cite{PS2}, \cite{PSS}, \cite{PSSW1}, \cite{PSSW2}, \cite{R},  \cite{Sz2}, \cite{PSSW3},  \cite{To} and \cite{CW} for some recent advances (however, this list of references is far from complete).   We also remark that if $M$ is toric, it turns out that the  stability condition
can be replaced by a simpler criterion involving the Futaki invariant,
and  the behavior of (\ref{KRF1}) is then well-understood \cite{WZ}, \cite{Zhu}.

There has also been much interest in understanding the behavior of the flow (\ref{KRF1}) with $\lambda=-1$ on manifolds with $c_1(M) \le 0$ (and not strictly definite).  In this case, smooth K\"ahler-Einstein metrics cannot exist and the K\"ahler class of $\omega(t)$ must degenerate in the limit.   If $M$ is a minimal model of general type, it is shown in \cite{Ts} and later generalized in \cite{TZha} that the K\"ahler-Ricci flow converges to the unique singular K\"ahler-Einstein metric.   If $M$ is not of general type, the K\"ahler-Ricci flow collapses and converges weakly to a generalized K\"ahler-Einstein metric if the canonical line bundle $K_M$ is semi-ample \cite{SoT1, SoT2}.

In contrast, the case when $c_1(M)$ is nonnegative or indefinite has  been little studied.  In complex dimension two, it is natural to consider the rational ruled surfaces,  known as  the Hirzebruch surfaces and denoted  $M_0, M_1, M_2, \ldots$  Indeed, all rational surfaces can be obtained from $\mathbb{P}^2$ and Hirzebruch surfaces via consecutive blow-ups.
 In this paper we describe several distinct behaviors of the K\"ahler-Ricci flow  on the manifolds $M_k$.  We consider the unnormalized K\"ahler-Ricci flow
\begin{eqnarray} \label{KRF0}
\ddt{} \omega  =  \mbox{} -  \textrm{Ric}(\omega), \quad \omega(0)  =  \omega_0.
\end{eqnarray}
for $\omega_0$ in any given K\"ahler class.  Write $\omega(t) = \frac{\sqrt{-1}}{2\pi} g_{i\ov{j}}dz^i \wedge d\overline{z^j}$, for $g=g(t)$ the K\"ahler metric associated to $\omega(t)$.
We find that, in the Gromov-Hausdorff sense, the flow $g(t)$ may:  shrink the manifold to a point, collapse to a lower dimensional manifold or contract a divisor on $M_k$.  As we will see, the particular outcome depends on $k$ and the initial K\"ahler class of $\omega_0$.  Much of this behavior was conjectured by Feldman, Ilmanen and Knopf in their detailed analysis \cite{FIK} of self-similar solutions of the Ricci flow.
   We confirm the Feldman-Ilmanen-Knopf conjectures, under the assumption that the initial metric is invariant under a maximal compact subgroup of the automorphism group.  

  Our results in this paper give some evidence that the K\"ahler-Ricci flow may indeed provide an analytic approach to the classification theory of algebraic varieties as suggested in \cite{SoT2}.
In general,  if the canonical line bundle $K_M$ is not nef, the unnormalized K\"ahler-Ricci flow (\ref{KRF0}) must become singular at some finite time, say $T$.
  If the limiting K\"ahler class is big as $t\rightarrow T$, a number of conjectures about the behavior of the flow have been made in  \cite{SoT1, SoT2, T2}.
  It is conjectured  that the limiting K\"ahler metric has a metric completion $(M', d_T)$, where  $M'$ is an algebraic variety obtained from $M$ by an  algebraic procedure such as a divisorial contraction or flip. It is further proposed  in \cite{T2, SoT3} that $(M', d_T)$ have mild singularities and $(M, \omega(t))$ should converge to $( M', d_T)$ in the sense of Gromov-Hausdorff.  Our main result in this paper confirms this speculation in the case of  $\mathbb{P}^2$ blown up at one point (with any initial K\"ahler class) and a family of higher-dimensional analogues of Hirzebruch surfaces  if the initial K\"ahler metric is invariant under a maximal compact subgroup of the automorphism group.  In addition, our results may be relevant to a recent conjecture of Tian  \cite{T2} that an algebraic manifold is birational to a Fano manifold if and only if the unnormalized K\"ahler-Ricci flow (suitably interpreted) becomes extinct in finite time.

The Hirzebruch surfaces $M_0, M_1, \ldots$ are projective bundles over $\mathbb{P}^1$ which can be described as follows.  Write $H$ and $\mathbb{C}_{\mathbb{P}^1}$ for the hyperplane line bundle and trivial line bundle respectively over $\mathbb{P}^1$.  Then we define the Hirzebruch surface $M_k$ to be
\begin{equation}
M_k = \mathbb{P} (H^k \oplus \mathbb{C}_{\mathbb{P}^1}).
\end{equation}
One can check that $M_0$ and $M_1$ are the only  Hirzebruch surfaces with positive first Chern class.  $M_0$ is the manifold $\mathbb{P}^1 \times \mathbb{P}^1$ and  will not be dealt with in this paper (see instead  \cite{PS1}, \cite{P2}, \cite{TZhu},  \cite{PSSW2}, \cite{Zhu}).  $M_1$ can be identified with $\mathbb{P}^2$ blown up at one point.  It is already known that the normalized K\"ahler-Ricci flow (\ref{KRF1}) with $\lambda=1$ on $M_1$ starting at a toric metric $\omega_0$ in $c_1(M_1)$ converges to a K\"ahler-Ricci soliton after modification by automorphisms (see \cite{Zhu} and also \cite{PSSW3}).  However, the manifold $M_1$ is still of interest to us since we are considering the more general case of the initial metric $\omega_0$ lying in  \emph{any}  K\"ahler class.

Assume now that $k\ge 1$ and denote by $D_{\infty}$ the divisor in $M_k$ given by the image of the section $(0,1)$ of $H^{k} \oplus \mathbb{C}_{\mathbb{P}^1}$.  Since the complex manifold $M_k$ can also be described by $\mathbb{P}(  \mathbb{C}_{\mathbb{P}^1} \oplus H^{-k})$ we can define another divisor $D_0$  on $M_k$ to be that given by the image of the  section $(1,0)$ of $\mathbb{C}_{\mathbb{P}^1} \oplus H^{-k}$.

All of the Hirzebruch surfaces $M_k$ admit K\"ahler metrics.  Indeed, the cohomology classes of the line bundles $[D_0]$ and $[D_{\infty}]$ span $H^{1,1}(M; \mathbb{R})$ and every K\"ahler class $\alpha$ can be  written uniquely as
\begin{equation} \label{alpha}
\alpha = \frac{b}{k} [D_{\infty}] - \frac{a}{k} [D_0]
\end{equation}
for constants $a$, $b$ with $0< a < b$.  If $\alpha_t$ denotes  the K\"ahler class of a solution  $\omega(t)$ of the flow (\ref{KRF0})
 then a short calculation shows that the associated constants $a_t, b_t$ satisfy
 \begin{equation} \label{atbtn2}
 b_t = b_0 - t(k+2) \quad  \textrm{and} \quad   a_t = a_0 + t(k-2).
 \end{equation}

Our goal is to understand the behavior of the K\"ahler-Ricci flow with initial metric $\omega_0$ in any given K\"ahler class $\alpha$.
We focus on the case when $\omega_0$  is invariant under the action of a maximal compact subgroup $G_k \cong U(2)/\mathbb{Z}_k$ of the automorphism group of $M_k$.  We will say that $\omega_0$ satisfies the \emph{Calabi symmetry condition}.
 This symmetry  is explained in detail in Section \ref{calabi} and was used by Calabi \cite{Cal} to construct extremal K\"ahler metrics on $M_k$ (see also \cite{Sz1}).  Our first result shows that under this symmetry condition we can describe the convergence of the flow $(M, g(t))$ in the sense of Gromov-Hausdorff.

\pagebreak[3]
\begin{theorem} \label{thm1}  On the Hirzebruch surface $M_k$,
let $\omega = \omega(t)$ be a solution of the K\"ahler-Ricci flow (\ref{KRF0}) with initial K\"ahler metric $\omega_0$ satisfying the Calabi symmetry condition.  Assume that $\omega_0$ lies in the K\"ahler class $\alpha_0$ given by $a_0, b_0$ satisfying $0< a_0< b_0$.
Then we have the following:
\begin{enumerate}
\item[(a)]  If $k \ge 2$ then the flow (\ref{KRF0}) exists on $[0,T)$ with $T= (b_0-a_0)/2k$ and $(M_k, g(t))$ converges to $(\mathbb{P}^1, a_T  g_{\emph{FS}})$ in the Gromov-Hausdorff sense as $t \rightarrow T$, where $g_{\emph{FS}}$ is the Fubini-Study metric and $a_T$ is the constant given by (\ref{atbtn2}).
\item[(b)] If $k =1$ there are three subcases.
\begin{enumerate}
\item[(i)]  If $b_0=3a_0$ then the flow (\ref{KRF0}) exists on $[0,T)$ with $T=a_0$ and $(M_1, g(t))$ converges to a point in the Gromov-Hausdorff sense as $t \rightarrow T$.
\item[(ii)] If $b_0< 3a_0$ then the flow (\ref{KRF0}) exists on $[0,T)$ with $T=(b_0-a_0)/2k$ and, as in (a) above, $(M_1, g(t))$ converges to $(\mathbb{P}^1, a_T  g_{\emph{FS}})$ in the Gromov-Hausdorff sense as $t \rightarrow T$.
\item[(iii)] If $b_0>3a_0$ then the flow (\ref{KRF0}) exists on $[0,T)$ with $T=a_0$.  On compact subsets of $M_1\setminus D_0$, $g(t)$ converges smoothly to a K\"ahler metric $g_T$.  If
 $(\overline{M}, d_T)$ denotes the metric completion of  $(M_1 \setminus D_0, g_T)$, then
 $(M_1,g(t))$ converges to  $(\overline{M}, d_T)$ in the Gromov-Hausdorff sense as $t \rightarrow T$.     $(\overline{M}, d_T)$ has finite diameter and is homeomorphic to the manifold $\mathbb{P}^2$.   \end{enumerate}
\end{enumerate}
\end{theorem}

We now make some remarks about this theorem.  As mentioned above, the manifold $M_1$ can be identified with $\mathbb{P}^2$ blown up at one point.
The case (b).(i) occurs precisely when the initial K\"ahler form $\omega_0$ lies in the first Chern class $c_1(M_1)$.  This situation has been well-studied and the convergence result of (b).(i)  is an immediate consequence of the diameter bound of Perelman for the normalized K\"ahler-Ricci flow \cite{P2}, \cite{SeT}.
In the case (b).(iii), we use the   work of \cite{Ts},  \cite{TZha}, \cite{Zha} to obtain the smooth convergence of the metric outside $D_0$.  Our result shows that
the K\"ahler-Ricci flow `blows down' the exceptional curve on $M_1$.
 For more details about this  see Section \ref{sectionGH}.

We see from the above that, assuming the Calabi symmetry,  there are three distinct behaviors of the K\"ahler-Ricci flow on a Hirzebruch surface, depending on $k$ and the initial K\"ahler class:
\begin{itemize}
\item The $\mathbb{P}^1$ fiber collapses (cases (a) and (b).(ii))
\item The manifold shrinks to a point (case (b).(i))
\item The exceptional divisor is contracted (case (b).(iii)).
\end{itemize}

We can say some more about the cases (a) and (b).(ii) when the fiber collapses.  If $D_H$  denotes any fiber of the map $\pi: M_k \rightarrow \mathbb{P}^1$ then the line bundle associated to $D_H$ is given by $[D_H]= \pi^*H$.  The cohomology class of $D_H$ is represented by the smooth (1,1) form $\chi = \pi^* \ofs$ where $\ofs$ is the Fubini-Study metric on $\mathbb{P}^1$.  We can show that in the case $k \ge 2$ or $k=1$ with $b_0<3a_0$, the K\"ahler form $\omega(t)$ along the flow converges to $a_T \chi$ in a certain weak sense which we now explain.  Define for $0\le t < T = (b_0-a_0)/2k$ a reference K\"ahler metric
\begin{equation}
\hat{\omega}_t = a_t \chi + \frac{(b_t-a_t)}{2k} \theta,
\end{equation}
in $\alpha_t$, where $\theta$ is a certain closed nonnegative (1,1) form in $2[D_{\infty}]$ (see Lemma \ref{theta} below).  Observe that $\hat{\omega}_t$ converges to $a_T \chi$ as $t\rightarrow T$.  Now define a potential function $\tilde{\varphi}=\tilde{\varphi}(t)$ by
\begin{equation}
\omega(t) = \hat{\omega}_t + \frac{\sqrt{-1}}{2\pi} \partial \ov{\partial} \tilde{\varphi}(t),
\end{equation}
where $\tilde{\varphi}$ is subject to a normalization condition $\tilde{\varphi}|_{\rho=0}=0$ (see Section \ref{sectioncalabi}).  Then we have:

\begin{theorem} \label{thm2} Assume that $k\ge 2$ or $k=1$ and $b_0< 3a_0$.
Let $\omega(t)$ be a solution of the flow (\ref{KRF0}) on $M_k$ with $\omega_0$ satisfying the Calabi symmetry condition.  Then for all $\beta$ with $0< \beta<1$,
\begin{enumerate}
\item[(i)] $\tilde{\varphi}(t)$ tends to zero in $C^{1,\beta}_{\hat{g}_0}(M_k)$ as $t \rightarrow T$.
\item[(ii)]  For any compact set $K \subset M_k \setminus (D_{\infty} \cup D_0)$, $\tilde{\varphi}(t)$ tends to zero in $C^{2, \beta}_{\hat{g}_0}(K)$ as $t \rightarrow T$.  In particular, on such a compact set $K$, $\omega(t)$ converges to $a_T  \chi$ on $C^{\beta}_{\hat{g}_0}(K)$ as $t \rightarrow T$.
\end{enumerate}
\end{theorem}

In addition, we can extend our results to higher dimensions by considering  $\mathbb{P}^1$ bundles over $\mathbb{P}^{n-1}$ for $n \ge 2$.
 Write $H$ and $\mathbb{C}_{\mathbb{P}^{n-1}}$ for the hyperplane line bundle and trivial line bundle respectively over $\mathbb{P}^{n-1}$.
Then we define the $n$-dimensional complex manifold
$M_{n,k}$ by
\begin{equation}
M_{n,k} = \mathbb{P} (H^k \oplus \mathbb{C}_{\mathbb{P}^{n-1}}). \label{Mnk}
\end{equation}
We can define divisors $D_0$ and $D_{\infty}$  in the same manner as for $M_k$ above.   The cohomology classes $[D_0]$ and $[D_{\infty}]$ again span $H^{1,1}(M; \mathbb{R})$ (see for example \cite{GH} or \cite{IS}) and the K\"ahler classes are described as in (\ref{alpha}).  Similarly, we have a Calabi symmetry condition, where the maximal compact subgroup of the automorphism group is now  $G_k \cong U(n)/\mathbb{Z}_k$ (see Section \ref{calabi}).
We have the following generalization of Theorem \ref{thm1}.

\begin{theorem} \label{thm3}  On $M_{n,k}$,
let $\omega = \omega(t)$ be a solution of the K\"ahler-Ricci flow (\ref{KRF0}) with initial K\"ahler metric $\omega_0$ satisfying the Calabi symmetry condition.  Assume that $\omega_0$ lies in the K\"ahler class $\alpha_0$ given by $a_0, b_0$ satisfying $0< a_0< b_0$.
Then we have the following:
\begin{enumerate}
\item[(a)]  If $k \ge n$ then the flow (\ref{KRF0}) exists on $[0,T)$ with $T= (b_0-a_0)/2k$ and $(M_{n,k}, g(t))$ converges to $(\mathbb{P}^{n-1}, a_T g_{\emph{FS}})$ in the Gromov-Hausdorff sense as $t \rightarrow T$.
\item[(b)] If $1 \le k \le n-1$ there are three subcases.
\begin{enumerate}
\item[(i)]  If $a_0(n+k) = b_0(n-k)$ then the flow (\ref{KRF0}) exists on $[0,T)$ with $T=a_0/(n-k)$ and $(M_{n,k}, g(t))$ converges to a point in the Gromov-Hausdorff sense as $t \rightarrow T$.
\item[(ii)] If $a_0(n+k) > b_0(n-k)$ then the flow (\ref{KRF0}) exists on $[0,T)$ with $T=(b_0-a_0)/2k$ and, as in (a) above, $(M_{n,k}, g(t))$ converges to $(\mathbb{P}^{n-1}, a_T g_{\emph{FS}})$ in the Gromov-Hausdorff sense as $t \rightarrow T$.
\item[(iii)] If $a_0(n+k) < b_0(n-k)$ then the flow (\ref{KRF0}) exists on $[0,T)$ with $T=a_0/(n-k)$.  On compact subsets of $M_{n,k} \setminus D_0$, $g(t)$ converges smoothly to a K\"ahler metric $g_T$.  If
 $(\overline{M}, d_T)$ denotes the metric completion of  $(M_{n,k} \setminus D_0, g_T)$, then
 $(M_{n,k},g(t))$ converges to  $(\overline{M}, d_T)$ in the Gromov-Hausdorff sense as $t \rightarrow T$.     $(\overline{M}, d_T)$ has finite diameter and is homeomorphic to the  orbifold $\mathbb{P}^{n}/\mathbb{Z}_k$ (see Section \ref{orbifold}).
\end{enumerate}
\end{enumerate}
\end{theorem}

We also prove an analog of Theorem \ref{thm2} in higher dimensions (see Theorem \ref{thmconv} below).
Since Theorem \ref{thm3} includes Theorem \ref{thm1} as a special case, we will prove all of our results in this paper in the general setting of complex dimension $n$.  We will often write $M$ for $M_{n,k}$.

Finally, we mention some known results about K\"ahler-Ricci solitons on these manifolds.  For $ 1\leq k \leq n-1$, the manifold $M_{n,k}$ has positive first Chern class and admits  a K\"ahler-Ricci soliton  \cite{Koi, Cao2}.  In addition, K\"ahler-Ricci solitons have been constructed on the orbifolds $\mathbb{P}^{n}/\mathbb{Z}_k$ for $2\leq k\leq n-1$ \cite{FIK} and, in the noncompact case,  on line bundles over $\mathbb{CP}^{n-1}$  \cite{Cao2, FIK}.  The limiting behavior of such solitons is studied and used in \cite{FIK} to construct examples of extending the Ricci flow through singularities.
Theorem \ref{thm3}  shows in particular that if the initial K\"ahler is not proportional to $c_1(M_{n,k})$, the K\"ahler-Ricci flow on $M_{n,k}$ ($1\leq k \leq n$) will not converge to a K\"ahler-Ricci soliton on the same manifold after normalization.

The outline of the paper is as follows.  In Section \ref{background}, we describe some background material including the details of the Calabi ansatz.
In Section \ref{genestimates} we prove some estimates for the K\"ahler-Ricci flow on the manifolds $M_k$ which hold without any symmetry condition.  We then impose the Calabi symmetry assumption in Section \ref{sectioncalabi} to give stronger estimates, and in particular, we prove
Theorem \ref{thm2} (cf. Theorem \ref{thmconv}).  In Section \ref{sectionGH} we give proofs of the Gromov-Hausdorff convergence of the flow, thus establishing Theorems \ref{thm1} and \ref{thm3}.

\section{Background} \label{background}

\subsection{The anti-canonical bundle}

Let $M=M_{n,k}$ be the manifold given by (\ref{Mnk}).   The anti-canonical bundle $K_M^{-1}$ of $M$ can be described as follows.  If $\pi : M \rightarrow \mathbb{P}^{n-1}$ is the bundle map then  write $D_H = \pi^{-1} (H_{n-1})$ where $H_{n-1}$ is a fixed hyperplane in $\mathbb{P}^{n-1}$.
Then the anti-canonical line bundle $K^{-1}_{M}$  is given by
 \begin{equation} \label{canonical}
K^{-1}_{M} =  2[D_{\infty}] - (k-n) [D_H] =  \frac{(k+n)}{k} [D_{\infty}] + \frac{(k-n)}{k} [D_0].
 \end{equation}
and we have
\begin{equation}
 k [D_H] = [D_{\infty}] - [D_0].
 \end{equation}

\subsection{The Calabi ansatz} \label{calabi}

We now briefly describe the ansatz of \cite{Cal} following, for the most part, Calabi's exposition.  We use coordinates $(x_1, \ldots, x_n)$ on $\mathbb{C}^n\setminus \{ 0 \}$.  Then the manifold $\mathbb{P}^{n-1}= (\mathbb{C}^n\setminus \{0\})/\mathbb{C}^*$ can be described by $n$ coordinate charts $U_{1}, \ldots, U_n$, where $U_{i}$ is characterized by $x_{i} \neq 0$.  For a fixed $i$, the holomorphic coordinates $z^{j}_{(i)}$ on $U_{i}$, for $1 \le j \le n$, $j \neq i$ are given by $z^{j}_{(i)} = x_{j}/x_{i}$.  Then $M=M_{n,k}$ can be defined as the $\mathbb{P}^1$ bundle over $\mathbb{P}^{n-1}$  with a projective fiber coordinate $y_{(i)}$  on $\pi^{-1}(U_{i})$ which transforms by
\begin{equation}
y_{(\ell)} = \left( \frac{x_{\ell}}{x_{i}} \right)^k y_{(i)}, \quad \textrm{on } \  \pi^{-1}(U_{i} \cap U_{\ell}),
\end{equation}
where  $\pi: M \rightarrow \mathbb{P}^{n-1}$ denotes the bundle map.
Then the divisors $D_0$ and $D_{\infty}$ are given by $y_{(i)}=0$ and $y_{(i)}= \infty$ respectively.  We parametrize $M \setminus (D_0 \cup D_{\infty})$ by a $k$-to-one map $\mathbb{C}^n \setminus \{ 0 \} \rightarrow M \setminus (D_0 \cup D_{\infty})$ described as follows.  The point $(x_1, \ldots, x_n)$, with $x_{i}\neq 0$ say, maps to the point in $M \setminus (D_0 \cup D_{\infty}) \cap \pi^{-1}(U_{i})$ with coordinates $z^{j}_{(i)}= x_{j}/x_{i}$, $y_{(i)} = x_{i}^k$, for $j \neq i$.

It is shown in \cite{Cal} that the group $G_k \cong U(n)/\mathbb{Z}_k$  is a maximal compact subgroup of the automorphisms of $M$ via the natural action on $\mathbb{C}^n\setminus \{0\}$.  Moreover, any K\"ahler metric $g_{i \ov{j}}$ on $M$ which is invariant under $G_k$ is described on $\mathbb{C}^n \setminus \{ 0 \}$ as $g_{i \ov{j}} = \partial_i \partial_{\ov{j}} u$ for a potential function $u= u (\rho)$, where
\begin{equation} \label{rho}
\rho = \log \left( \sum_{i=1}^n |x_i|^2 \right).
\end{equation}

The potential function $u$ has to satisfy certain properties in order to define a K\"ahler metric.  Namely,
a K\"ahler metric $g_{i \ov{j}}$ with K\"ahler form $\omega = \frac{\sqrt{-1}}{2\pi} g_{i \ov{j}}dz^i \wedge d\ov{z^j}$ on $M$ in the class (see (\ref{alpha}))
\begin{equation} \label{alpha2}
\alpha = \frac{b}{k} [D_{\infty}] - \frac{a}{k} [D_0]
\end{equation}
 is given by the potential function $u: \mathbb{R} \rightarrow \mathbb{R}$ with $u'>0$, $u''>0$ together with the following asymptotic condition.  There exist  smooth functions $u_0, u_{\infty}: [0,\infty) \rightarrow \mathbb{R}$ with $u'_0(0)>0$, $u_{\infty}'(0)>0$ such that
\begin{equation} \label{u0}
u_0(e^{k\rho}) = u(\rho) -a\rho, \quad u_{\infty}(e^{-k\rho}) = u(\rho)-b\rho,
\end{equation}
for all $\rho \in \mathbb{R}$.
It follows that
$$ \lim_{\rho \rightarrow - \infty} u'(\rho) = a < b= \lim_{\rho \rightarrow \infty} u'(\rho).$$
Note that the divisor $D_0$ corresponds to $\rho = - \infty$ while $D_{\infty}$ corresponds to $\rho=\infty$.
The K\"ahler metric $g_{i \ov{j}}$ associated to $u$  is given in the $x_i$ coordinates by
\begin{equation} \label{metric}
g_{i \ov{j}} = \partial_i \partial_{\ov{j}} u = e^{-\rho} u'(\rho) \delta_{i j} + e^{-2\rho} \ov{x}_i x_j (u''(\rho) - u'(\rho)).
\end{equation}
Conversely, a K\"ahler metric $g$ determines the function $u$ up to the addition of a constant.

The metric $g$ has determinant
\begin{equation} \label{det}
\det g = e^{-n \rho} (u'(\rho))^{n-1} u''(\rho).
\end{equation}
Thus, if we define
\begin{equation} \label{Riccipotential}
v = - \log \det g = n \rho - (n-1) \log u'(\rho) - \log u''(\rho)
\end{equation}
then the Ricci curvature tensor $R_{i \ov{j}} = \partial_i \partial_{\ov{j}} v$ is given by
\begin{equation} \label{Ricci}
R_{i \ov{j}} = e^{-\rho} v'(\rho) \delta_{ij} + e^{-2\rho} \ov{x}_i x_{j} (v''(\rho) - v'(\rho)).
\end{equation}

Finally, we construct a reference metric $\hat{\omega}$ in the class $\alpha$ given by (\ref{alpha2}).  Define a potential function $\hat{u}$ by
\begin{equation} \label{hatu}
\hat{u}(\rho) = a \rho + \frac{(b-a)}{k} \log (e^{k\rho}+1).
\end{equation}
Then one can check by the above definition that the associated K\"ahler form $\hat{\omega}$ lies in the class $\alpha$.  We observe in addition that $\hat{\omega}$ can be decomposed into a sum of nonnegative (1,1) forms.
The smooth (1,1) form $\chi = \pi^*\ofs \in [D_H]$
 is represented by  the straight line function $u_{\chi}$:
\begin{equation} \label{eqnchi}
u_{\chi}(\rho) = \rho, \quad \chi = \frac{\sqrt{-1}}{2\pi} \partial \ov{\partial} u_{\chi}.
\end{equation}
In addition,
let $u_{\theta}$ and $\theta$ be respectively the potential and associated closed (1,1)-form defined by
\begin{equation} \label{eqntheta}
u_{\theta} = 2 \log(e^{k\rho}+1), \quad \theta = \frac{\sqrt{-1}}{2\pi} \partial \ov{\partial} u_{\theta}.
\end{equation}
The form $\theta$ lies in the cohomology class $2[D_{\infty}]$.  Moreover,
\begin{equation}
\hat{u} = a u_{\chi} + \frac{(b-a)}{2k} u_{\theta}, \quad \textrm{and} \quad
\hat{\omega}= a \chi + \frac{(b-a)}{2k} \theta \in \alpha.
\end{equation}

The following lemma will be useful later.

\begin{lemma} \label{theta}
The smooth nonnegative closed (1,1) form $\theta$ in  $2[D_{\infty}]$ given by (\ref{eqntheta}) satisfies
$$ \chi^{n-1} \wedge \theta >0 \quad \textrm{and} \quad \int_M \theta^n >0.$$
\end{lemma}

\begin{proof}
 Note that $$\chi^{n-1} \wedge \theta  = \frac{2k}{b-a} \chi^{n-1} \wedge \hat{\omega} >0.$$ Also,  from the construction of $\theta$  and the formula (\ref{det}), we have $\int_M \theta^n >0$.
  \qed
\end{proof}

\subsection{The orbifold $\mathbb{P}^n/\mathbb{Z}_k$} \label{orbifold}

Using homogeneous coordinates $Z_1, \ldots, Z_{n+1}$,  we let $\mathbb{P}^n/\mathbb{Z}_k$ be the weighted projective space invariant under
$$j \cdot [Z_1, \ldots, Z_{n+1}] = [e^{2\pi j \sqrt{-1}/k} Z_0,  \ldots, e^{2\pi j \sqrt{-1}/k}Z_n,  Z_{n+1}],$$
for $j=0,1,2, \ldots, k-1$.
The space $\mathbb{P}^n/\mathbb{Z}_k$ has a natural orbifold structure, branched over the point $[0,\ldots, 0,1]$.  In addition, there is a holomorphic map $f: M_{n,k} \rightarrow \mathbb{P}^n/\mathbb{Z}_k$ given as follows.  A point in $\pi^{-1}(U_i)$ with coordinates $z_{(i)}^j$ (for $j \neq i$) and fiber coordinate $y_{(i)}$ maps  to the point in $\mathbb{P}^n/\mathbb{Z}_k$ with homogeneous coordinates $$[z_{(i)}^1, \ldots, z_{(i)}^{i-1}, 1, z_{(i)}^{i+1}, \ldots, z_{(i)}^n, y_{(i)}^{-1} ],$$
or $[0, \ldots, 0,1]$ if $y_{(i)}=0$.
The map $f$ is well-defined because of the group action.   Note that the inverse image $f^{-1}([0,\ldots, 0,1])$ is the divisor $D_0$ and
\begin{equation}
f|_{M_{n,k}\setminus D_0} : M_{n,k}\setminus D_0 \rightarrow (\mathbb{P}^n/\mathbb{Z}_k) \setminus \{ [0,\ldots, 0,1] \}
\end{equation}
is an isomorphism.  In the case $n=2$, $k=1$, the divisor $D_0$ is the exceptional curve on $\mathbb{P}^2$ blown up at one point and $f: M_{2,1} \rightarrow \mathbb{P}^2$ is the blow-down map.

Finally we note that there is a map $\mathbb{C}^n  \rightarrow (\mathbb{P}^n/\mathbb{Z}_k)\setminus \{ Z_{n+1}=0\}$ given by
$$(x_1, \ldots, x_n) \mapsto [x_1, \ldots, x_n, 1].$$
This map is $k$-to-one on $\mathbb{C}^n \setminus \{ 0 \}$.
With respect to these coordinates (and the corresponding coordinates $(x_1, \ldots, x_n)$ on $M_{n,k}$ as described above)  one can check that $f$ is the identity map on $\mathbb{C}^{n}\setminus \{0\}$.

\section{Estimates for the K\"ahler-Ricci flow} \label{genestimates}

In this section we prove some  estimates for a solution $\omega(t)$ to the K\"ahler-Ricci flow (\ref{KRF0}) on $M=M_{n,k}$ that hold \emph{without} the assumption of  Calabi symmetry.  If   the initial metric $\omega_0$ lies in the class $\alpha_0$ given by constants $0<a_0 <b_0$ then
 the K\"ahler class $\alpha_t$ of $\omega(t)$ evolves by
\begin{equation}
\alpha_t = \frac{b_t}{k} [D_{\infty}] - \frac{a_t}{k} [D_0]= \frac{(b_t-a_t)}{k} [D_{\infty}] + a_t [D_H],
\end{equation}
where
\begin{equation} \label{btat}
b_t = b_0 - (k+n)t \quad \textrm{and} \quad a_t = a_0 + (k-n)t.
\end{equation}

Define
 \begin{equation}
 T = \sup\{ t\geq 0 \ | \ \alpha_t \textrm{ is K\"ahler} \}.
 \end{equation}
Note that the K\"ahler metric $\hat{\omega}_t$ given by \begin{equation} \label{ref1}
\hat{\omega}_t=a_t \chi + \frac{(b_t-a_t)}{2k} \theta
 \end{equation}
 for $t \in [0,T)$ and $\theta$ from Lemma \ref{theta} lies in $\alpha_t$.

We observe that the K\"ahler-Ricci flow exists on $[0,T)$.

\begin{theorem}
There exists a unique smooth solution of the K\"ahler-Ricci flow (\ref{KRF0}) on $M$  starting with $\omega_0 \in \alpha_0$  for $t$ in  $[0, T)$.
\end{theorem}
\begin{proof}
This follows from a general and well-known result in the K\"ahler-Ricci flow.  Indeed, let $X$ be any K\"ahler manifold and $\alpha_0$ be a K\"ahler class on $X$ with $\omega_0 \in \alpha_0$.  If
$$T = \sup \{ t \ge 0 \ | \  \alpha_0 + t [K_X] >0 \},$$
then it is shown in \cite{Cao1}, \cite{Ts}, \cite{TZha} that there is a unique smooth solution $\omega=\omega(t)$ of the K\"ahler-Ricci flow (\ref{KRF0}) on $X$ starting at $\omega_0$, for $t$ in $[0,T)$.\qed
\end{proof}

We now deal with the behavior of the flow as $t \rightarrow T$ in various different cases.

\subsection{The case $k \ge n$} \label{kgen}

In this case, the class $\alpha_t$ remains K\"ahler for $0< t < T$ where
$$T = \frac{b_0-a_0}{2k}.$$  As $t \rightarrow T$, the difference $b_t - a_t$ tends to zero, while the constant $a_t$ remains bounded below away from zero.

The reference metric $\hat{\omega}_t$ in $\alpha_t$ is given by
 \begin{equation} \label{ref2}
\hat{\omega}_t=a_t \chi + \frac{(b_t-a_t)}{2k} \theta =a_t \chi + (T-t) \theta  \in \alpha_t.
\end{equation}
From (\ref{canonical}) we see that the closed (1,1) form $\theta - (k-n) \chi$ lies in the first Chern class $c_1(M)$.  Hence
 there is a smooth volume form $\Omega$  on $M$ such that $$ \ddbar\log \Omega = -\theta + (k-n)\chi.$$
We consider the  parabolic Monge-Amp\`ere equation:

\begin{equation}\label{flow2}
 \ddt{\varphi} = \log \frac{ (\hat{\omega}_t +\ddbar \varphi)^n}{(T-t) \Omega}, \qquad \varphi|_{t=0}= \varphi_0,
\end{equation}
where $\hat{\omega}_0 + \frac{\sqrt{-1}}{2\pi}  \partial \ov{\partial} \varphi_0= \omega_0 \in \alpha_0$.  If $\varphi=\varphi(t)$ solves (\ref{flow2}) then
$$\omega(t) = \hat{\omega}_t + \frac{\sqrt{-1}}{2\pi} \partial \ov{\partial} \varphi$$
solves the K\"ahler-Ricci flow (\ref{KRF0}).

Note that in the following two lemmas, we only use the assumption $k \ge n$ to obtain a  uniform lower bound of the constant $a_t$ away from zero, for $t \in [0,T)$.

\begin{lemma} \label{lemmavolbound}
There exists a constant $C$ depending only on the initial data such that $$|\varphi(t)|\leq C, \quad \omega^n(t) \le C \Omega.$$

\end{lemma}

\begin{proof}

By Lemma \ref{theta}, since $$\hat{\omega}_t^n= (a_t \chi + (T-t)\theta )^n \ge n(T-t) a_t^{n-1} \chi^{n-1} \wedge   \theta,$$
there exist constants $C_1, C_2>0$ independent of $t$ such that
\begin{equation} \label{eqnvolref}
C_1(T-t) \Omega \le \hat{\omega}_t^n  \le C_2(T-t) \Omega.
\end{equation}
Note that we are making use of the fact that $a_t$ is bounded from below away from zero.
To obtain the upper bound of $\varphi$, consider the evolution of $\psi= \varphi - (1+\log C_2)t$.  We claim that $\sup_{M \times [0,T)} \psi = \sup_M \psi|_{t=0}$.  Otherwise there exists a point  $(x,t)  \in M \times (0,T)$ at which $\partial \psi/\partial t \ge 0$ and $\ddbar \psi \le 0$.  Thus, at that point,
$$0\le \frac{\partial \psi}{\partial t} \le \log \frac{\hat{\omega}_t^n}{(T-t)\Omega} - 1 - \log C_2  \le -1,$$ a contradiction.
Hence $\sup_{M \times [0,T)} \psi = \sup_M \psi|_{t=0}$ and thus $\varphi$ is uniformly bounded from above.
A lower bound on $\varphi$ is obtained similarly.

For the upper bound of $\omega^n$, we will bound $H= \log \frac{\omega^n}{\Omega} - A \varphi$, where $A$ is a constant to be determined later.  Writing $\tr_{\omega}{\omega'} = n\frac{ \omega^{n-1} \wedge \omega'}{\omega^n}$ where $\omega'$ is any $(1,1)$-form, we compute
\begin{eqnarray*}
\frac{\partial}{\partial t} H & = & \Delta \dot{\varphi} + \tr_{\omega} \left( {\ddt{} \hat{\omega}_t} \right) - A \dot{\varphi},
\end{eqnarray*}
where $\Delta$ denotes the Laplace operator associated to $g(t)$.
Since $\dot{\varphi} = H + A\varphi - \log (T-t)$ and $\theta \ge 0$, we have
\begin{eqnarray*}
\frac{\partial}{\partial t} H & = & \Delta H + A \Delta \varphi + (k-n) \tr_{\omega} \chi - \tr_{\omega} \theta - A(H+A \varphi - \log (T-t)) \\
& \le & \Delta H + An - A a_t \tr_{\omega} \chi + (k-n) \tr_{\omega} \chi - AH - A^2 \varphi + A \log T.
\end{eqnarray*}
Choosing $A$ sufficiently large so that $A a_t \ge (k-n)$ and using the fact that $\varphi$ is uniformly bounded, we see that $H$ is bounded from above by the maximum principle.
\qed
\end{proof}

We have, in addition, the following estimate.

\begin{lemma} \label{trchi} There exists a uniform constant $C>0$ such that
$$\tr_{\omega} \chi = n\frac{\omega^{n-1} \wedge \chi}{\omega^n} \leq C.$$
\end{lemma}

\begin{proof} This is a `parabolic Schwarz lemma' similar to the one given in \cite{SoT1}.  We use the
 maximum principle.
 Let $\ofs = \frac{\sqrt{-1}}{2\pi} h_{\alpha \ov{\beta}} dz^{\alpha} \wedge d\ov{z}^{\beta}$ be the Fubini-Study metric on $\mathbb{P}^{n-1}$ and let $\pi: M \rightarrow \mathbb{P}^{n-1}$ be the bundle map.
We will calculate the evolution of
\begin{equation}
w=\tr_{g}(\pi^*h)=g^{i\overline{j}}\pi^{\alpha}_i
\pi^{\overline{\beta}}_{\overline{j}}h_{\alpha\overline{\beta}} = n \frac{\omega^{n-1} \wedge \chi}{\omega^n}.
\end{equation}
A standard computation shows that
\begin{eqnarray}\label{sch} \nonumber
\Delta w &=& g^{k\overline{l}}\partial_k
\partial_{\overline{l}} \left( g^{i\overline{j}}\pi^{\alpha}_i
\pi^{\overline{\beta}}_{\overline{j}}h_{\alpha\overline{\beta}} \right)\\
&=&g^{i\overline{l}}g^{k\overline{j}}R_{k\overline{l}}
\pi^{\alpha}_{i}
\pi^{\overline{\beta}}_{\overline{j}}h_{\alpha\overline{\beta}}+
g^{i\overline{j}}g^{k\overline{l}}\pi^{\alpha}_{i,k}\pi^{\overline{\beta}}_{\overline{j},\overline{l}}h_{\alpha\overline{\beta}}
-g^{i\overline{j}}g^{k\overline{l}}S_{\alpha\overline{\beta}
\ga\overline{\de}} \pi^{\alpha}_i
\pi^{\overline{\beta}}_{\overline{j}}\pi^{\ga}_k
\pi^{\overline{\de}}_{\overline{l}},
\end{eqnarray}
where $S_{\alpha\overline{\beta} \ga\overline{\de}}$ is the
curvature tensor of $h_{\alpha\bar{\beta}}$.
By the definition of $w$ we have
\begin{eqnarray} \label{deltaw}
\Delta w \ge g^{i\overline{l}}g^{k\overline{j}}R_{k\overline{l}}
\pi^{\alpha}_i
\pi^{\overline{\beta}}_{\overline{j}}h_{\alpha\overline{\beta}} +
g^{i\overline{j}}g^{k\overline{l}}\pi^{\alpha}_{i,k}\pi^{\overline{\beta}}_{\overline{j},\overline{l}}h_{\alpha\overline{\beta}}
  - 2w^2.
\end{eqnarray}
Now
\begin{eqnarray} \nonumber
\ddt{w}&=&-g^{i\overline{l}}g^{k\overline{j}}\ddt{g_{k\overline{l}}}
\pi^{\alpha}_i
\pi^{\overline{\beta}}_{\overline{j}} h_{\alpha\overline{\beta}}\\ \label{dwdt}
&=&g^{i\overline{l}}g^{k\overline{j}}R_{k\overline{l}}
\pi^{\alpha}_i
\pi^{\overline{\beta}}_{\overline{j}}h_{\alpha\overline{\beta}}.
\end{eqnarray}
Combining (\ref{deltaw}) and (\ref{dwdt}),
we have
\begin{equation}
(\ddt{} - \Delta) w \le - g^{i\overline{j}}g^{k\overline{l}}\pi^{\alpha}_{i,k}\pi^{\overline{\beta}}_{\overline{j},\overline{l}}h_{\alpha\overline{\beta}}+ 2w^2.
\end{equation}
  On the other hand, by a standard argument (see \cite{Y1} for example)
\begin{equation}
\frac{|\nabla w|^2}{w} \le g^{i\overline{j}}g^{k\overline{l}}\pi^{\alpha}_{i,k}\pi^{\overline{\beta}}_{\overline{j},\overline{l}}h_{\alpha\overline{\beta}}
\end{equation}
and thus
\begin{equation} \label{logw}
(\ddt{}- \Delta) \log w \leq 2w.
\end{equation}

We now compute the evolution of the quantity  $ L  = \log w - A \varphi$, where $A$ is a constant to be determined later. From the arithmetic-geometric means inequality $$\frac{ \lambda_1 + \cdots + \lambda_n}{n} \ge ( \lambda_1 \cdots \lambda_n)^{1/n} , \quad \textrm{for  } \ \lambda_1, \ldots,  \lambda_n \ge 0,$$
and (\ref{eqnvolref})
we have,
\begin{eqnarray} \label{gam}
\tr_{\omega} \hat{\omega}_t \ge n
  \left( \frac{\hat{\omega}_t^n}{\omega^n} \right)^{1/n}
 \ge  c \left( \frac{(T-t) \Omega}{\omega^n}\right)^{1/n},
\end{eqnarray}
for a uniform constant $c>0$. Then, using the inequality $\tr_{\omega} \hat{\omega}_t \ge a_t w$ together with (\ref{logw}) and (\ref{gam}),
\begin{eqnarray*}
 \left( \ddt{}-\Delta\right) L
&\leq& 2w -A \log \left( \frac{\omega^n}{(T-t)\Omega} \right) + An - A \tr_{\omega} \hat{\omega}_t \\
&\leq& -w + A \log \left( \frac{(T-t)\Omega}{\omega^n} \right) + An - c  \left( \frac{(T-t) \Omega}{\omega^n}\right)^{1/n},
\end{eqnarray*}
where we have chosen $A$ sufficiently large so that $(A-1)a_t \ge 3$.
Note that the function $\mu \mapsto A\log \mu - c\mu^{1/n}$ for $\mu >0$ is uniformly bounded from above.
Hence if the maximum of $L$ is achieved at a point $(x_0, t_0) \in M \times (0,T)$  then, at that point, $w$ is uniformly bounded from above.
Since we have already shown in Lemma \ref{lemmavolbound} that $\varphi$ is uniformly bounded along the flow, the required upper bound of $w$ follows by the maximum principle.
\qed
\end{proof}

\subsection{The case $1 \le k \le n-1$} \label{sectionKRk1}

There are three distinct types of behavior here which depend on the choice of initial K\"ahler class $\alpha_0$.  We deal with each in turn.

\subsubsection{The subcase $a_0(n+k) = b_0(n-k)$.} \label{subcaseperelman}

In this case $\alpha_0 = (a_0/(n-k))c_1(M)$ and the class $\alpha_t = (a_0/(n-k)-t) c_1(M)$ is proportional to the first Chern class.  After renormalizing, this is the K\"ahler-Ricci flow (\ref{KRF1}) with $\lambda=1$ on a  manifold with $c_1(M)>0$.  It is shown in \cite{FIK} that $M$  admits a K\"ahler-Ricci soliton and we refer the reader to the results of \cite{TZhu} and also \cite{PSSW3}.  For our purposes, we only need the fact   that the diameter of the metric is bounded along the normalized K\"ahler-Ricci flow  \cite{P1}, \cite{SeT}.

\subsubsection{The subcase $a_0(n+k) > b_0(n-k)$.}

In this case the K\"ahler class $\alpha_t$ remains K\"ahler until time $T= (b_0-a_0)/2k$.  We have $\lim_{t \rightarrow T} b_t = \lim_{t \rightarrow T} a_t>0$ as in the case $k\ge n$.  Lemmas \ref{lemmavolbound} and \ref{trchi} hold with the same proofs in this case since, as pointed out in Section \ref{kgen}, we only used there the fact that $a_t$ is uniformly bounded from below away from zero.

\subsubsection{The subcase $a_0(n+k) < b_0(n-k)$.}

In this case the K\"ahler class $\alpha_t$ remains K\"ahler until time $T=a_0/(n-k)$.   As $t \rightarrow T$, $a_t$ tends to zero while $b_t$ remains bounded below away from zero.
  The metrics $\hat{\omega}_t$ and classes $\alpha_t$ satisfy the following properties for all $t \in [0,T)$:
\begin{enumerate}
\item[(1)]
 The limit $\displaystyle{\hat{\omega}_T = \lim_{t \rightarrow T} \hat{\omega}_t}$ is a smooth closed nonnegative (1,1) form satisfying
$\int_M \hat{\omega}_T^n >0.$
\item[(2)] For all $\varepsilon>0$ sufficiently small (independent of $t$), the class $\alpha_t - \varepsilon [D_0]$ is K\"ahler.
\end{enumerate}

Indeed, (1)  follows from Lemma \ref{theta} and (2) is immediate from the definition of $\alpha_t$.
We will use these properties to prove the following (cf. \cite{Ts}, \cite{TZha}, \cite{Zha}).

\begin{theorem} \label{theoremTZ} Let $\omega(t)$ be a solution of the K\"ahler-Ricci flow on $M$ starting at $\omega_0$ in $\alpha_0$ with $a_0(n+k) < b_0(n-k)$.  Then:
\begin{enumerate}
\item[(i)] If $\Omega$ is a fixed volume form on $M$ then there exists a uniform constant $C$ such that
$$\omega^n(t) \le C\Omega, \quad \emph{for all } t \in [0,T).$$
\item[(ii)]
There exists a closed semi-positive (1,1) current $\omega_T$ on $M$ which is smooth outside the exceptional curve $D_0$ and has an $L^\infty$-bounded local potential such that following holds.  The metric $\omega(t)$ along the K\"ahler-Ricci flow converges in $C^{\infty}$ on compact subsets of $M\setminus D_0$ to $\omega_T$ as $t \rightarrow T$.
\end{enumerate}
\end{theorem}

\begin{proof}  This result is essentially contained in \cite{TZha}, \cite{Zha}, but we will include an outline of the proof for the reader's convenience.  First, define a closed (1,1) form $\eta$ by
$$\eta = \ddt{} \hat{\omega}_t = (k-n) \chi - \theta,$$
so that the reference metric $\hat{\omega}_t$ is given by
 $\hat{\omega}_t = \hat{\omega}_0 + t \eta$. Since $-\eta$ is in $c_1(M)$ there exists a volume form $\Omega$ on $M$ with $\ddbar\log \Omega = \eta$.
Let $\varphi=\varphi(t)$ be a solution of the parabolic Monge-Amp\`ere equation
\begin{equation} \label{pma}
 \ddt{\varphi} = \log \frac{ (\hat{\omega}_t + \frac{\sqrt{-1}}{2\pi} \partial\ov{\partial} \varphi)^n}{\Omega},
\end{equation}
with $\varphi|_{t=0} = 0$, for $t \in [0,T)$.   Then $\omega(t) = \hat{\omega} +\frac{\sqrt{-1}}{2\pi} \partial\ov{\partial} \varphi$ solves the K\"ahler-Ricci flow (\ref{KRF0}).  We will first bound $\varphi$.
Notice that $\varphi$ is uniformly bounded from above by a simple maximum principle argument. To obtain a uniform $L^{\infty}$ bound on $\varphi$ we will show that $\dot{\varphi}$ is uniformly bounded from above for $t \in [T/2, T)$.
Compute
 \begin{equation}
 (\ddt{} - \Delta) \dot{\varphi} = \textrm{tr}_{\omega} \eta.
 \end{equation}
 Then
 \begin{equation}
 (\ddt{} - \Delta)(t \dot{\varphi}- \varphi -nt) =  - \textrm{tr}_{\omega} \hat{\omega}_0 \le 0,
 \end{equation}
 using the fact that $\Delta \varphi = n- \textrm{tr}_{\omega}(\hat{\omega}_0 + t \eta)$.
 It follows from the maximum principle that $t \dot{\varphi}$ is uniformly bounded from above.    Hence $\dot{\varphi}$ is uniformly bounded from above for $t$ in $[T/2,T)$.

Rewrite (\ref{pma}) as
 \begin{equation} \label{ma}
 (\hat{\omega}_t + \ddbar \varphi)^n = e^{\dot{\varphi}} \Omega.
 \end{equation}
By property (1) above,  the K\"ahler metric $\hat{\omega}_t$ satisfies $\int_M \hat{\omega}_t^n >c>0$ for a uniform constant $c$  for all $t \in [T/2, T)$ and the limit $\hat{\omega}_T$ is a smooth nonnegative (1,1) form.
Hence we can  apply the results of \cite{Kol}, \cite{Zha}, \cite{EGZ} on the complex Monge-Amp\`ere equation to obtain an $L^{\infty}$ bound on $\varphi$ which is uniform in $t$.  Note that part (i) of the Theorem follows from the upper bound of $\dot{\varphi}$.  For (ii), we use property (2) of $\alpha_t$ as listed above.  The argument of Tsuji \cite{Ts} (cf. \cite{Y1}) gives  second order estimates for $\varphi$, depending on the $L^{\infty}$ estimate, on all compact sets of $M \setminus D_0$.  The rest of the theorem follows by standard theory. \qed
\end{proof}

\section{Estimates under the Calabi symmetry condition} \label{sectioncalabi}

In this section we assume that the initial metric $\omega_0$ satisfies the symmetry condition of Calabi.  We prove estimates for the solution of the K\"ahler-Ricci flow and in particular we give a proof of Theorem \ref{thm2} (see Theorem \ref{gest} below.)

Let $\omega(t)$ be a solution of the K\"ahler-Ricci flow (\ref{KRF0}) on a time interval $[0,T)$.  The K\"ahler-Ricci flow can be described in terms of the potential $u = u(\rho,t)$.  Noting that $\omega$ determines $u$ only up to the addition of a constant, we consider $u=u(\rho, t)$ solving
\begin{equation}\label{uKRF}
\ddt{} u(\rho, t) =  \log u''(\rho,t) + (n-1) \log u'(\rho,t) -n\rho + c_t,
\end{equation}
where
\begin{equation} \label{ct}
c_t = - \log u''(0,t) - (n-1) u'(0,t).
\end{equation}
The notation $u'(\rho,t)$ denotes the partial derivative $(\partial u/\partial \rho) (\rho,t)$ and similarly for $u''(\rho,t)$.
We assume that $\rho \mapsto u(\rho,0)$ represents the initial K\"ahler metric $g_0$ and we impose the further normalization condition that
\begin{equation} \label{ic}
u(0,0)=0.
\end{equation}
Then by (\ref{Riccipotential}), the solution $g=g(t)$ of (\ref{KRF0}) can be written as $g_{i \ov{j}} = \partial_i \partial_{\ov{j}} u$.
  The constant $c_t$ is chosen so that $\ddt{} u(0,t)=0$ and
thus  $u(0,t)=0$ on $[0,T)$.  The existence of a unique solution of the K\"ahler-Ricci flow on $[0,T)$ and the parabolicity of (\ref{uKRF}) ensures the existence of a smooth unique $u(\rho,t)$ solving (\ref{uKRF}).

The K\"ahler metric at time $t$ lies in the cohomology class $$\alpha_t = \frac{b_t}{k}[D_{\infty}] - \frac{a_t}{k} [D_0]$$ along the flow, where $a_t$ and $b_t$ are given by (\ref{btat}). We have
$$\lim_{\rho \rightarrow -\infty} u'(\rho,t) = a_t, \quad \lim_{\rho \rightarrow \infty} u'(\rho, t) = b_t$$
and by convexity
$$a_t < u'(\rho, t) < b_t, \quad \textrm{for all } \rho \in \mathbb{R}.$$
Next,  the evolution equations for $u'$, $u''$ and $u'''$ are given by
\begin{eqnarray} \label{upevolution}
\ddt{} u' & = & \frac{u'''}{u''} + \frac{(n-1) u''}{u'} -n \\ \label{udpevolution}
\ddt{} u'' & = & \frac{u^{(4)}}{u''} - \frac{(u''')^2}{(u'')^2} + \frac{(n-1)u'''}{u'} - \frac{(n-1) (u'')^2}{(u')^2} \\  \nonumber
\ddt{} u''' & =  & \frac{u^{(5)}}{u''} - \frac{3u''' u^{(4)}}{(u'')^2} + \frac{2(u''')^3}{(u'')^3}+ \frac{(n-1) u^{(4)}}{u'} \\  \label{utpevolution}
&& \mbox{} - \frac{3(n-1) u'' u'''}{(u')^2} + \frac{2(n-1) (u'')^3}{(u')^3},
\end{eqnarray}
as can be seen from differentiating (\ref{uKRF}).

For the rest of this section we assume that $u=u(\rho,t)$ solves the K\"ahler-Ricci flow (\ref{uKRF}) with  (\ref{ct}) and (\ref{ic}).

\subsection{The case $k \ge n$} \label{calabisubsectionk2}

The following elementary lemma shows that as $t \rightarrow T=(b_0-a_0)/2k$, the potential $u$ converges pointwise to the  function $a_T u_{\chi}$.

\begin{lemma}  \label{lemmapointwise}
The function $u=u(\rho,t)$ satisfies, for all $\rho$ in $\mathbb{R}$,
\begin{enumerate}
\item[(i)] $\displaystyle{0 < u'(\rho,t) - a_t < 2k(T-t)},$ for all $t \in [0,T)$;
\item[(ii)] $\displaystyle{\lim_{t \rightarrow T} (u(\rho,t)- a_T \rho) =0}$.
\end{enumerate}
\end{lemma}
\begin{proof}
(i) follows immediately from the convexity of $u$ and the definition of $a_t$ and $b_t$.  For (ii), recall that $u(0,t)=0$ for all $t \in [0,T)$ and so
$$u(\rho,t)-a_t \rho = \int_0^{\rho} (u'(s,t) - a_t)ds$$
Applying (i),
$$\left| u(\rho,t)-a_t   \rho \right| \le 2k(T-t) | \rho | \longrightarrow 0,$$
as $t \rightarrow T$, while $a_t \rightarrow a_T$.  \qed
\end{proof}

As a simple application of this we prove a lower bound for the K\"ahler metric along the flow.

\begin{lemma} \label{lowerbd} Along the K\"ahler-Ricci flow, we have
$$\omega(t)  \ge a_t \chi$$
for all $t$ in $[0,T)$.
\end{lemma}

\begin{proof}
With a slight abuse of notation, we write $\chi = \frac{\sqrt{-1}}{2\pi} \chi_{i \ov{j}} dz^i \wedge d\ov{z}^j$.    Then by (\ref{metric}),
\begin{equation} \label{chi}
g_{i\bar{j}}(t) = e^{-\rho} u' \delta_{ij} + e^{-2\rho}  \ov{x}_i x_j (u''-u'), \quad \chi_{i \ov{j}} =  e^{-\rho} \delta_{ij} - e^{-2\rho} \ov{x}_i x_j.
\end{equation}
Since $u''>0$ and $u' >a_t$,
$$g_{i \ov{j}}(t) \ge u' e^{-\rho} \left( \delta_{ij} - \frac{\ov{x}_i x_j}{\sum_k |x_k|^2} \right) \ge a_t e^{-\rho} \left( \delta_{ij} - \frac{\ov{x}_i x_j}{\sum_k |x_k|^2} \right) = a_t \chi_{i \ov{j}},$$
as required.  \qed

\end{proof}

We have the following further estimates for $u$ which will give an upper bound for the metric $\omega(t)$.

\begin{lemma} \label{lemmaest1}  There exists a constant $C$ depending only on the initial data such that
\begin{equation} \label{udp}
0<u''(\rho, t)  \le C \min \left( \frac{e^{k \rho}}{(1+e^{k \rho})^2},  (T-t) \right)
\end{equation}
and
\begin{equation} \label{utp}
|u'''(\rho,t) | \le Cu''(\rho,t)
\end{equation}
for all $(\rho,t) \in \mathbb{R} \times [0,T)$.
\end{lemma}
\begin{proof}
We begin by establishing the bound
\begin{equation}\label{bd1}
u''(\rho, t) \le  C \frac{e^{k \rho}}{(1+e^{k \rho})^2}.
\end{equation}
Let $\hat{g}_0$ be the reference metric with associated potential  $\hat{u}_0$ (see (\ref{ref1})).  Then from (\ref{det}) we see that
 \begin{eqnarray} \nonumber
\det{\hat{g}_0} & = &   e^{-n \rho} ( \hat{u}_0'(\rho))^{n-1}  \hat{u}_0''(\rho) \\ \nonumber
& =  &  k (b_0-a_0) e^{- n \rho} \left( a_0 + (b_0-a_0) \frac{e^{k\rho}}{(1+ e^{k \rho})} \right)^{n-1} \frac{e^{k\rho}}{(1+e^{k\rho})^2} \\
& \le & C e^{-n \rho} \frac{e^{k\rho}}{(1+e^{k\rho})^2}. \label{bd2}
 \end{eqnarray}
By Lemma \ref{lemmavolbound}
 the volume form $\omega^n$ is uniformly bounded from above by a fixed volume form along the flow, and hence $\det g(t) \le C \det \hat{g}_0$ for some uniform constant $C$.  Combining this fact with (\ref{det}) and (\ref{bd2}), we have
 \begin{equation} \label{eqnvolformbd}
 ( u'(\rho,t) ) ^{n-1} u''(\rho,t) \le C\frac{e^{k\rho}}{(1+e^{k\rho})^2},
 \end{equation}
Then  (\ref{bd1}) follows, since $u'(\rho)$ is uniformly bounded from below away from zero.

Next, we give a proof of the bound (\ref{utp}).  We will apply the maximum principle to the quantity $u'''/u''$.  Before we do this, observe that for each fixed $t \in [0,T)$,
\begin{equation}
\label{limits1}
\lim_{\rho \rightarrow - \infty} \frac{u'''(\rho,t)}{u''(\rho,t)} = k, \quad \lim_{\rho \rightarrow \infty} \frac{u'''(\rho,t)}{u''(\rho,t)} = -k.
\end{equation}
Indeed, for the first limit, we can compute $u'''/u''$ in terms of the function $s \mapsto u_0(s, t)$ associated to $u$.  From (\ref{u0}),
\begin{eqnarray} \nonumber
u(\rho,t) & =&  a_t \rho + u_0(e^{k\rho}, t) \\ \nonumber
u'(\rho,t) & = & a_t + k e^{k\rho} u'_0(e^{k\rho}, t) \\ \nonumber
u''(\rho,t) & = & k^2 e^{k\rho} u'_0(e^{k\rho}, t) + k^2e^{2k\rho} u''_0(e^{k\rho}, t) \\ \label{uppp}
u'''(\rho,t) & = & k^3 e^{k\rho} u'_0(e^{k\rho},t) + 3k^3 e^{2k \rho} u''_0(e^{k\rho},t) + k^3 e^{3k \rho} u_0'''(e^{k\rho},t).
\end{eqnarray}
Hence
\begin{eqnarray*}
\frac{u'''(\rho,t)}{u''(\rho,t)} & = & \frac{k u'_0(e^{k\rho},t) + 3k e^{k \rho} u''_0(e^{k\rho},t) + k e^{2k \rho} u_0'''(e^{k\rho},t)}{u'_0(e^{k\rho}, t) + e^{k\rho} u''_0(e^{k\rho},t)} \longrightarrow k, \ \textrm{as } \ \rho \rightarrow -\infty,
\end{eqnarray*}
since $u'_0(0, t) >0$.  The second limit can be proved in a similar way using the function $u_{\infty}$.

Using (\ref{udpevolution}) and (\ref{utpevolution}) we find that $u'''/u''$ evolves by
\begin{eqnarray*}
\ddt{} \left( \frac{u'''}{u''} \right) & = & \frac{1}{u''} \left( \frac{u^{(5)}}{u''} - \frac{3u''' u^{(4)}}{(u'')^2} + \frac{2(u''')^3}{(u'')^3} + \frac{(n-1)u^{(4)}}{u'}  - \frac{3(n-1)u''u'''}{(u')^2} + \frac{2(n-1)(u'')^3}{(u')^3} \right) \\
&& \mbox{} - \frac{u'''}{(u'')^2} \left( \frac{u^{(4)}}{u''} - \frac{(u''')^2}{(u'')^2} + \frac{(n-1)u'''}{u'} - \frac{(n-1)(u'')^2}{(u')^2} \right).
\end{eqnarray*}
To obtain an upper bound of $u'''/u''$ we argue as follows.  From (\ref{limits1}) we may assume without loss of generality that the maximum of $u'''/u''$ occurs at an interior point $(\rho_1, t_1) \in \mathbb{R} \times (0,T)$.  At that point,
$$\left( \frac{u'''}{u''} \right)' = \frac{u^{(4)}}{u''} - \frac{(u''')^2}{(u'')^2} = 0$$
and
$$\left( \frac{u'''}{u''} \right)'' = \frac{u^{5}}{u''} - \frac{3u'''u^{(4)}}{(u'')^2} + 2 \frac{(u''')^3}{(u'')^3}  \le 0.$$
Then at $(\rho_1, t_1)$,
$$\ddt{} \left( \frac{u'''}{u''} \right) \le - \frac{2(n-1) u'''}{(u')^2} + \frac{2(n-1)(u'')^2}{(u')^3}$$
and hence
$$\left( \frac{u'''}{u''} \right) (\rho_1, t_1) \le \left(\frac{u''}{u'}\right)(\rho_1, t_1) \le C,$$
using (\ref{bd1}).  This gives the upper bound for $u'''/u''$.  The argument for the lower bound is similar. In fact, since $u''>0$ we obtain $u'''/u'' \ge -k$. This establishes (\ref{utp}).

It remains to prove the estimate $u'' \le C(T-t)$.  Fix $t$ in $[0,T)$.  Since the bound (\ref{bd1}) implies that $u''(\rho)$ tends to zero  as $\rho$ tends to  $\pm \infty$, there exists $\tilde{\rho} \in \mathbb{R}$ such that
$$u''(\tilde{\rho}) = \sup_{\rho \in \mathbb{R}} u''(\rho).$$
By the Mean Value Theorem and (\ref{utp}) we have, for all $\rho \in \mathbb{R}$,
$$u''(\tilde{\rho})- u''(\rho)  \le C u''(\tilde{\rho}) |\rho - \tilde{\rho} |.$$
Then for $|\rho - \tilde{\rho}| \le 1/2C$,
$$u''(\rho) \ge \frac{u''(\tilde{\rho})}{2} ,$$
and hence
$$\frac{1}{2C} u''(\tilde{\rho})=   \int_{ |\rho - \tilde{\rho}| \le 1/2C} \frac{u''(\tilde{\rho}) }{2}d\rho < \int_{-\infty}^{\infty} u''(\rho) d\rho = b_t - a_t =2k(T-t).$$
The bound $u'' \le C(T-t)$ then follows. \qed
\end{proof}

We can now use these estimates on $u$ to obtain bounds on the K\"ahler metric $\omega(t)$ along the K\"ahler-Ricci flow.  In the following, $\omega(t)$ will always denote a solution of the K\"ahler-Ricci flow (\ref{KRF0}) with initial metric $\omega_0$ satisfying the Calabi symmetry.

\begin{theorem} \label{gest}
We have
\begin{enumerate}
\item[(i)] $\displaystyle{\sup_M \tr_{\hat{g}_0}{g} \le C}.$
\item[(ii)] For any compact set $K \subset M \setminus (D_{\infty} \cup D_0)$,
$$\sup_K | \nabla_{\hat{g}_0} g |_{\hat{g}_0} \le C_K.$$
\end{enumerate}
\end{theorem}
\begin{proof}
An elementary computation shows that
$$\tr_{\hat{g}_0}{g} = \frac{u''}{\hat{u}_0''} + (n-1) \frac{u'}{\hat{u}_0'}.$$
But we have $u'/\hat{u}_0' \le b_0/a_0$ and, making use of
 Lemma \ref{lemmaest1},
$$\frac{u''}{\hat{u}_0''}  = \frac{(1+e^{k\rho})^2}{k(b_0-a_0) e^{k\rho}} u'' \le C ,$$
and this gives (i).

We now prove (ii).  By (i) it suffices to bound
$$\frac{\partial}{\partial x_{k}} g_{i\ov{j}} = e^{-2\rho} (u''-u') (\ov{x}_k \delta_{ij} + \ov{x}_i \delta_{jk}) + e^{-3\rho} \ov{x}_i x_j \ov{x}_k (u'''-3u''+2u'),$$
in a given compact set $K  \subset M \setminus (D_{\infty} \cup D_0)$.
But $u''$ and $u'''$ and uniformly bounded from above by Lemma \ref{lemmaest1}, and
in $K$, the functions $\rho$ and $x_i$ are uniformly bounded.  This gives (ii).
\qed
\end{proof}

From Lemma \ref{lowerbd} and Theorem \ref{gest} we obtain the following immediate corollary.

\begin{corollary}
We have
$$\frac{1}{C} \le \emph{diam}_{g(t)} M \le C.$$
\end{corollary}

Moreover we  prove:

\begin{theorem} \label{thmconv}
Define $\tilde{\varphi}= \tilde{\varphi}(t)$ by
\begin{equation} \label{phitilde}
\omega(t) = \hat{\omega}_t + \ddbar \tilde{\varphi}, \quad \tilde{\varphi}|_{\rho=0} = 0.
\end{equation}
Then for all $\beta$ with $0< \beta<1$,
\begin{enumerate}
\item[(i)] $\tilde{\varphi}$ tends to zero in $C^{1,\beta}_{\hat{g}_0}(M)$ as $t \rightarrow T$.
\item[(ii)]  For any compact set $K \subset M \setminus  (D_{\infty} \cup D_0)$, $\tilde{\varphi}$ tends to zero in $C^{2, \beta}_{\hat{g}_0}(K)$ as $t \rightarrow T$.  In particular, on $K$, $\omega(t)$ converges to $a_T \chi$ on $C^{\beta}_{\hat{g}_0}(K)$ as $t \rightarrow T$.
\end{enumerate}
\end{theorem}
\begin{proof}
By the normalization of $\tilde{\varphi}$, we have $\tilde{\varphi}(t) = u(t) - \hat{u}_t$.  As $t \rightarrow T$, $\hat{u}_t$ tends to $a_T u_{\chi}$.  Then by   Lemma \ref{lemmapointwise}, $\tilde{\varphi}$ converges pointwise to zero.  Taking the trace of (\ref{phitilde}) with respect to $\hat{g}_0$, applying the first part of Theorem \ref{gest} and using the fact that $\tr_{\hat{g}_0}\hat{g}_t$ is bounded, we see that $\Delta_{\hat{g}_0} \tilde{\varphi}$ is uniformly bounded on $M$, giving (i).  The second part of Theorem \ref{gest} gives (ii).
  \qed
\end{proof}

Finally, we use the estimates of Lemma \ref{lemmaest1} together with the bound of Lemma \ref{trchi} to show that, away from the divisors $D_0$ and $D_{\infty}$, the fibers are collapsing.

\begin{theorem} \label{thmfiber} Let $\pi^{-1}(z)$ be the fiber of $\pi: M \rightarrow \mathbb{P}^{n-1}$ over the point $z\in \mathbb{P}^{n-1}$.  Define $\omega_z(t) = \omega(t)|_{\pi^{-1}(z)}$.  Then for any compact set $K \subset M \setminus (D_0 \cup D_{\infty})$, there exists a constant $C_K$ such that
\begin{equation} \label{fiberbd}
\sup_{z \in \mathbb{P}^{n-1}} \| \omega_z(t) \|_{C^0(\pi^{-1}(z) \cap K)} \le C_K(T-t).
\end{equation}
\end{theorem}
\begin{proof}
Fix a compact set $K \subset M \setminus (D_0 \cup D_{\infty})$.  From (\ref{det}) we see that on $K$, the quantity $\omega^n/\Omega$ is uniformly equivalent to $u''$.  Then
by Lemma \ref{lemmaest1}, there exists a constant $C_K$ such that
\begin{equation} \label{CK}
\omega^n(x) \le C_K(T-t) \Omega(x), \quad \textrm{for } x \in K.
\end{equation}
Now at a point $x \in K$, choose complex coordinates $z^1, \ldots, z^n$ so that $$\chi = \frac{\sqrt{-1}}{2\pi} \sum_{i=1}^{n-1} dz^i \wedge d\ov{z^i} \quad  \textrm{and} \quad \omega = \frac{\sqrt{-1}}{2\pi} \sum_{i=1}^{n} \lambda_i dz^i \wedge d\ov{z^i},$$
for $\lambda_1, \ldots, \lambda_n>0$.  The coordinate  $z^n$ is in the fiber direction and we wish to obtain an upper bound for $\lambda_n$.
From Lemma \ref{trchi} we have
$$\sum_{i=1}^{n-1} \frac{1}{\lambda_i} \le C.$$
Hence $\lambda_1, \ldots, \lambda_{n-1}$ are uniformly bounded from below away from zero and we have, for uniform constants $C, C'$,
$$\lambda_n \le C \frac{1}{\lambda_1 \cdots \lambda_{n-1}} \frac{\omega^n(x)}{\Omega} \le C' \cdot C_K (T-t)$$
from  (\ref{CK}).  This proves (\ref{fiberbd}).\qed
\end{proof}

As a corollary of this and Theorem \ref{gest}, the diameter of the fibers goes to zero as $t$ tends to $T$.

\begin{corollary} \label{corollaryfiber}
As above, let $\pi^{-1}(z)$ be the fiber of $\pi: M \rightarrow \mathbb{P}^{n-1}$ over the point $z\in \mathbb{P}^{n-1}$.   Then
$$\lim_{t \rightarrow T} \left( \sup_{z \in \mathbb{P}^{n-1}} \emph{diam}_{g(t)} \,\pi^{-1}(z) \right)=0.$$
\end{corollary}
\begin{proof}
Fix $\varepsilon >0$.   By Theorem \ref{gest}, (i) there exists a tubular neighborhood $N_{\varepsilon}$ of $D_0 \cup D_{\infty}$ such that for all $z \in \mathbb{P}^{n-1}$ and all $t \in [0,T)$,
\begin{equation} \label{epsilon1}
\textrm{diam}_{g(t)} ( \pi^{-1}(z) \cap N_{\varepsilon} ) < \frac{\varepsilon}{2}.
\end{equation}
On the other hand, applying Theorem \ref{thmfiber} with $K = M \setminus N_{\varepsilon}$ we see that for $t$ sufficiently close to $T$,
\begin{equation} \label{epsilon2}
\textrm{diam}_{g(t)} ( \pi^{-1}(z) \cap K ) < \frac{\varepsilon}{2},
\end{equation}
for all $z \in \mathbb{P}^{n-1}$.  Combining (\ref{epsilon1}) and (\ref{epsilon2}) completes the proof.  \qed

\end{proof}

\subsection{The case $1 \le k \le n-1$}

As in section \ref{sectionKRk1}, there are three subcases.  We do not require any further estimates when
 $a_0(n+k) = b_0(n-k)$  and so we move on to the other two cases.

\subsubsection{The subcase $a_0(n+k) > b_0(n-k)$.}

We obtain all the results of subsection \ref{calabisubsectionk2} by identical proofs.   The key point is that, in this case, $a_t$ is uniformly bounded from below away from zero.

\subsubsection{The subcase $a_0(n+k) < b_0(n-k)$.}

Recall that in this case, $a_t = a_0+ (k-n)t$ tends to zero as $t$ tends to the blow-up time $T$.  Note that by part (ii) of Theorem \ref{theoremTZ} we already have $C^{\infty}$ estimates for the metric $g(t)$ on $M \setminus D_0$.  In this subsection we will obtain estimates for the metric in a neighborhood of $D_0$.  First we have the following estimate on $u'$.

\begin{lemma} \label{lemmak1up}
There exists a uniform constant $C$ such that for all $t \in [0,T)$ and all $\rho \in \mathbb{R}$,
 $$ 0 <  u'(\rho, t) - a_t \leq  C e^{k\rho/n}.$$
\end{lemma}

\begin{proof}  The first inequality follows from the definition of $a_t$ and the convexity of $u$.  For the upper bound of $u'(\rho,t)-a_t$ we argue as follows.
By part (i) of Theorem \ref{theoremTZ} the volume form of $\omega(t)$ is uniformly bounded along the flow.  Then by the same argument as in the proof of (\ref{eqnvolformbd}), we have
$$ ( u'(\rho,t) )^{n-1} u''(\rho,t) \leq C \frac{e^{k\rho}}{(1+ e^{k\rho})^2} \le C e^{k\rho}.$$
Hence
$$ ((u'(\rho,t))^n)'\leq Ce^{k\rho}.$$
Integrating in $\rho$ we obtain
$$ (u'(\rho,t))^n - a_t^n \leq C e^{k\rho} $$
and thus
$$  u'(\rho, t) \leq a_t  +  C e^{ \frac{k}{n} \rho},$$
as required.
\qed
\end{proof}

Note that the conclusion of Lemma \ref{lemmak1up} could be strengthened for $\rho>0$.  However, in this subsection we need only concern ourselves with the case of negative $\rho$ since  the metric is bounded away from $D_0$.

\begin{lemma} \label{lemmaudpk1}  There exists a uniform constant $C$ such that
  $$ u'' \leq C (u'-a_t) (b_t - u').$$
In particular,
$$u'' \leq C e^{k\rho/n }.$$
\end{lemma}

\begin{proof}

We evolve the quantity $H = \log u'' - \log (u'-a_t) - \log (b_t- u').$  Using (\ref{upevolution}) and (\ref{udpevolution}) we compute
\begin{eqnarray} \nonumber
\frac{\partial H}{\partial t} & = & \frac{1}{u''} \left( \frac{u^{(4)}}{u''} - \frac{(u''')^2}{(u'')^2} + \frac{(n-1)u'''}{u'} - \frac{(n-1)(u'')^2}{(u')^2} \right) \\ \nonumber
&& \mbox{} - \frac{1}{u'-a_t} \left( \frac{u'''}{u''} + \frac{(n-1)u''}{u'} -k \right) \\ \label{evolveH}
&& \mbox{} - \frac{1}{b_t-u'} \left( - \frac{u'''}{u''} - \frac{(n-1)u''}{u'} -k \right).
\end{eqnarray}

Before applying the maximum principle we check that $H$ remains bounded from above as $\rho$ tends to $\pm \infty$.  For the case of $\rho$ negative we use  (\ref{uppp}) to obtain
\begin{eqnarray*}
\frac{u''(\rho,t)}{(u'(\rho,t)-a_t)(b_t -u'(\rho,t))} & = & \frac{k e^{k\rho} u'_0(e^{k\rho}, t) + ke^{2k\rho} u''_0(e^{k\rho}, t)}{e^{k\rho} u'_0(e^{k\rho},t) (b_t-a_t - ke^{k\rho}u'_0(e^{k\rho},t))} \\
& \le & \frac{k}{b_t-a_t - ke^{k\rho} u'_0(e^{k\rho},t)} + \frac{ke^{k\rho}u''_0(e^{k\rho},t)}{u'_0(e^{k\rho}, t) (b_t-a_t)} \\
& \le & \frac{k}{b_t-a_t} +1,
\end{eqnarray*}
as $\rho$ tends to $- \infty$, where we are using the fact that $u'_0(0,t)>0$.  Note that  $b_t-a_t$ remains bounded from below away from zero. Similarly, we can show that $H$ is bounded from above as $\rho$ tends to positive infinity.

Suppose then that $H$ has a maximum at a point $(\rho_0,t_0) \in  \mathbb{R} \times (0, T)$.  Then at this point, we have
\begin{equation} \label{dhz}
\frac{u'''}{u''} - \frac{u''}{u'-a_t} + \frac{u''}{b_t-u'} =0
\end{equation}
and
\begin{equation} \label{ddhz}
\frac{u^{(4)}}{u''} - \frac{(u''')^2}{(u'')^2} - \frac{u'''}{u'-a_t} + \frac{(u'')^2}{(u'-a_t)^2} + \frac{u'''}{b_t-u'} + \frac{(u'')^2}{(b_t-u')^2} \le 0.
\end{equation}
Combining (\ref{evolveH}), (\ref{dhz}) and (\ref{ddhz}) we see that at $(\rho_0, t_0)$,
$$0 \leq -u''\left( \frac{1}{(u'-a_t)^2} + \frac{1}{ (b_t - u')^2} \right) - \frac{(n-1)u''}{(u')^2} + \frac{k}{u'-a_t} + \frac{k}{b_t - u'}    $$ and hence
$$ \frac{u''}{(u'-a_t)(b_t -u')} \leq \frac{k(b_t - a_t) }{(u'-a_t)^2 + (b_t - u')^2}\leq C,$$
and the result follows by the maximum principle.
\qed
\end{proof}

\section{Gromov-Hausdorff convergence} \label{sectionGH}

In this section we prove Theorem \ref{thm3} (and hence also Theorem \ref{thm1}) on the Gromov-Hausdorff convergence of the K\"ahler-Ricci flow.  We assume in this section that $g(t)$ is a solution of the K\"ahler-Ricci flow (\ref{KRF0}) on $[0,T)$ with initial metric $g_0$ satisfying the Calabi symmetry condition.

We begin by recalling the definition of Gromov-Hausdorff convergence.  It will be convenient to use the characterization  given in, for example, \cite{F} or \cite{GW}.   Let $(X, d_X)$ and $(Y,d_Y)$ be two compact metric spaces.  We define the Gromov-Hausdorff distance $d_{\textrm{GH}}(X,Y)$ to be the infimum of all $\varepsilon>0$ such that the following holds.   There exist maps $F: X \rightarrow Y$ and $G: Y \rightarrow X$   such that
\begin{equation} \label{GH1}
|d_X(x_1, x_2) - d_Y (F(x_1), F(x_2))|\leq \varepsilon, \quad \textrm{for all } x_1, x_2 \in X
\end{equation} and
\begin{equation} \label{GH2}
d_X(x, G\circ F(x)) < \varepsilon, \quad \textrm{for all } x \in X
\end{equation}
and the two symmetric properties for $Y$ also hold.   Note that we do not require the maps $F$ and $G$ to be continuous.

We say that a sequence of compact metric spaces $(X_i, d_{X_i})$ converges to $(Y, d_Y)$ in the sense of Gromov-Hausdorff if $d_{\textrm{GH}}(X_i, Y)$ tends to zero as $i$ tends to infinity.

\subsection{The case $k\ge n$}

Again we consider first the case when $k \ge n$.  We prove the following.

\begin{theorem}
$(M, g(t)) $ converges to $(\mathbb{P}^{n-1}, a_T g_{\textrm{FS}})$ in the Gromov-Hausdorff sense as $t\rightarrow T$.

\end{theorem}

\begin{proof}  We write $d_{g(t)}$ and $d_{T}$ for the distance functions on $M$ and $\mathbb{P}^{n-1}$ associated to $g(t)$ and $a_T \omega_{\textrm{FS}}$ respectively.  Let $\ve>0$ be given.   Let $\sigma: \mathbb{P}^{n-1} \rightarrow M$ be any smooth map satisfying $\pi \circ \sigma = \textrm{id}_{\mathbb{P}^{n-1}}$ such that $\sigma(\mathbb{P}^{n-1})$ does not intersect $D_0 \cup D_{\infty}$.  We first verify that (\ref{GH2}) (and its analog for $Y$) hold whenever $t$ is sufficiently close to $T$.  We are taking here $(X, d_X) = (M, d_{g(t)})$, $(Y, d_Y) = (\mathbb{P}^{n-1}, d_{T})$, $F= \pi$ and $G=\sigma$ in the definition of Gromov-Hausdorff convergence.

For $x \in M$, by Corollary \ref{corollaryfiber},
$$d_{g(t)} (x, \sigma\circ \pi(x)) \leq \textrm{diam}_{g(t)} (\pi^{-1}(\pi(x))) \longrightarrow 0$$ uniformly in $x$ as $t\rightarrow T$.
Moreover, for any $y$ in $\mathbb{P}^{n-1}$, $ d_{\textrm{FS}} ( y,  \pi \circ\sigma(y)) =0$ holds trivially.

Now we verify (\ref{GH1}).   First, again by Corollary \ref{corollaryfiber}, we choose $t$ close enough to $T$ so that for all $z$ in $\mathbb{P}^{n-1}$,
\begin{equation} \label{diam}
\textrm{diam}_{g(t)} ( \pi^{-1}(z)) < \varepsilon/4.
\end{equation}
For any $x_1, x_2 \in M$, let $y_i = \pi(x_i) \in \mathbb{P}^{n-1}$. Let $\gamma$ be a geodesic in $\mathbb{P}^{n-1}$ such that $d_{T}(y_1, y_2) = L_{a_T g_{\textrm{FS}}}(\gamma),$
where $L_{ a_T g_{\textrm{FS} }}( \gamma ) $ is the arc length of $\gamma$ with respect to $a_T g_{\textrm{FS}}$.
Choose a small tubular neighborhood $N$ of $D_0 \cup D_\infty$ so that $\sigma(\mathbb{P}^{n-1})$ does not intersect $N$.  Then $\tilde{\gamma}= \sigma \circ \gamma$ is a smooth path in $M \setminus N$ joining the points  $x'_1= \sigma(y_1) $ and $x_2'= \sigma(y_2)$.
By Theorem \ref{thmconv}, $g(t)$ converges to $a_T \pi^*g_{\textrm{FS}}$ uniformly on $M \setminus N$, and hence for $t$ sufficiently close to $T$,
\begin{equation} \label{L}
L_{g(t)}(\tilde{\gamma}) < L_{a_T g_{\textrm{FS}}}(\gamma) + \varepsilon/2.
\end{equation}
Then from (\ref{diam}) and (\ref{L}),
\begin{equation} \label{GH11}
 d_{g(t)}(x_1, x_2) \leq d_{g(t)}(x_1', x_2')+ \frac{\varepsilon}{2}  \leq  L_{g(t)}(\tilde{\gamma}) +\frac{\varepsilon}{2}
  \leq    L_{a_T g_{\textrm{FS}}}(\gamma) +\varepsilon = d_{T}(y_1, y_2) + \varepsilon.
 \end{equation}
On the other hand, by Lemma \ref{lowerbd},
\begin{equation} \label{GH12}
d_{g(t)}(x_1, x_2) \geq  \left(\frac{a_t}{a_T}\right)^{1/2} d_{T}( y_1, y_2),
\end{equation}
and $a_t \rightarrow a_T$ as $t \rightarrow T$.
Combining (\ref{GH11}) and (\ref{GH12}) gives
$$|d_{g(t)} (x_1, x_2) - d_{T}(\pi (x_1), \pi(x_2)) |\leq \varepsilon,$$
and hence the first case of (\ref{GH1}).

The second case of (\ref{GH1}) is simpler.   Let $y_1, y_2$  be in  $\mathbb{P}^{n-1}$ and write $x_i = \sigma (y_i)$.  Since $\sigma(\mathbb{P}^{n-1})$ does not intersect $D_0 \cup D_{\infty}$, we can apply Theorem \ref{thmconv} to obtain the following  convergence uniformly in $t$ and the choice of $y_1, y_2$, $x_1$, $x_2$:
$$\lim_{t \rightarrow T} d_{g(t)} (x_1, x_2) =  d_{T} (y_1, y_2),$$
as required.
\qed\end{proof}

\subsection{The case $1 \le k \le n-1$.}

If $a_0(n+k) > b_0(n-k)$ then one can apply verbatim the argument as in the case $k \geq n$ and the K\"ahler-Ricci flow collapses to $\mathbb{P}^{n-1}$.

If $a_0(n+k) = b_0(n-k)$, then $\alpha_0 $ is proportional to the first Chern class $c_1(M)$.  As discussed in subsection \ref{subcaseperelman},
 the diameter is uniformly bounded along  the normalized K\"ahler-Ricci flow with initial K\"ahler metric in $c_1(M)$. Then after scaling, the diameter tends to $0$ along the unnormalized K\"ahler-Ricci flow.

We consider then just the subcase $a_0(n+k) < b_0(n-k)$.  We will show that as $t \rightarrow T$ the divisor $D_0$ in $M$ contracts. First, we have the following lemma.

\begin{lemma} \label{lemmagup} There exists a uniform constant $C$ such that the metric $g_{i\ov{j}}=g_{i\ov{j}}(t)$ on $\mathbb{C}^n\setminus \{0\}$ satisfies the estimate
\begin{equation} \label{gub}
g_{i\ov{j}}(t) \le a_t \chi_{i\ov{j}} + Ce^{(k-n)\rho/n} \delta_{ij},
\end{equation}
where $\chi_{i\ov{j}} = e^{-\rho} \delta_{ij} - e^{-2\rho} x_i x_j $.
\end{lemma}

\begin{proof}  This follows from Lemmas \ref{lemmak1up} and \ref{lemmaudpk1}.  Indeed,
\begin{eqnarray*}
g_{i\ov{j}} & = & e^{-\rho} u' \delta_{ij} + e^{-2\rho} \ov{x}_i x_j (u'' - u') \\
& \le & C e^{(k-n)\rho/n} \delta_{ij} + e^{-\rho} a_t \delta_{ij} + Ce^{(k-n)\rho/n} e^{-\rho} \ov{x}_i x_j- e^{-2\rho} \ov{x}_i x_j a_t \\
& \le & a_t \chi_{i\ov{j}} + Ce^{(k-n)\rho/n} \delta_{ij} ,
\end{eqnarray*}
since $u'(\rho, t) > a_t$. \qed
\end{proof}

Recall that by Theorem \ref{theoremTZ}, the metric $g_t$ along the K\"ahler-Ricci flow converges in $C^{\infty}$ on compact subsets of $M \setminus D_0$ to a singular metric $g_T$ which is smooth on $M \setminus D_0$.
We will apply this and Lemma \ref{lemmagup} to prove the following result on the metric completion of the manifold $(M \setminus D_0, g_T)$.

\begin{theorem} \label{thmcompletion} Let $g_T$ be the smooth metric on $M \setminus D_0$ obtained by $$ g_T = \lim_{t\rightarrow T} g(t),$$
and let $(\overline{M}, d)$ be the completion of the Riemannian manifold $(M \setminus D_0, g_T)$ as a metric space.  Then $(\overline{M},d)$ is a metric space with finite diameter and is homeomorphic to the orbifold $\mathbb{P}^n/\mathbb{Z}_k$ (see Section \ref{orbifold}).
\end{theorem}

\begin{proof}
Let $f: M \rightarrow \mathbb{P}^n/\mathbb{Z}_k$ be the holomorphic map described in Section \ref{orbifold}.  Recall that $f$ restricted to $M \setminus D_0$ is represented in the $(x_1, \ldots, x_n)$ coordinates by the identity map $\textrm{id}: \mathbb{C}^n\setminus \{0\} \rightarrow \mathbb{C}^n \setminus \{ 0 \}$.  Now $f$ is an isomorphism on $M \setminus D_0$  and Theorem \ref{theoremTZ} implies that $g(t)$ converges locally in $C^\infty(M \setminus
D_0).$  Thus
it only remains to check the limiting behavior of $g(t)$ near $D_0$ as $t\rightarrow T$ (that is, in a neighborhood of the origin in $\mathbb{C}^n$).

We make a simple observation.  Suppose $g_{ij}$ is a continuous Riemannian metric on $B \setminus \{0\}$, where $B = \{ (x_1, \ldots, x_n) \in \mathbb{C}^n \ | \ |x_1|^2 + \cdots + |x_n|^2 \le 1\}$.  Suppose in addition that $g_{ij}$ satisfies the inequality
$$g_{ij} \le \frac{C}{r^{\beta}} \delta_{ij},$$
for some $\beta <2$, where $r = ( |x_1|^2 + \cdots + |x_n|^2)^{1/2}$.  Then the completion of $(B \setminus \{ 0 \}, g)$ as a metric space has finite diameter and is homeomorphic  to $B$ with  topology induced from $\mathbb{C}^n$.

Now from Lemma \ref{lemmagup}, since $a_t \rightarrow 0$ as $t \rightarrow T$, we see that on $\mathbb{C}^n\setminus \{ 0 \}$,
$$(g_T)_{i\ov{j}} \le \frac{C}{r^{2(n-k)/n}} \delta_{ij}$$
and hence the theorem follows from the observation above with $\beta = 2(n-k)/n$. \qed
\end{proof}

In addition, the proof of Theorem \ref{thmcompletion} gives:

\begin{lemma} Let $N_\varepsilon$ be an $\varepsilon$-tubular neighborhood of $D_0$ in $M$ with respect to the fixed metric $\hat{g}_0$. Then
$$ \lim_{\varepsilon\rightarrow 0} \limsup_{t \rightarrow T} \emph{diam}_{g(t)} N_{\varepsilon} = 0.$$
\end{lemma}

We can then prove:

\begin{theorem}  $(M, g(t))$ converges to $(\overline{M}, d)$ in the Gromov-Hausdorff sense as $t \rightarrow T$, where $(\overline{M}, d)$ is the metric space as described in Theorem \ref{thmcompletion}.
\end{theorem}

\begin{proof}  This is a simple consequence of the results described above.  Identifying $\overline{M}$ with $\mathbb{P}^n/\mathbb{Z}_k$ as in the proof of Theorem \ref{thmcompletion},
  let $F: M \rightarrow \overline{M}$ be the map corresponding to $f:M \rightarrow \mathbb{P}^n/\mathbb{Z}_k$.
Let $G: \overline{M} \rightarrow M$ be any map satisfying $F \circ G = \textrm{id}_{\overline{M}}$.  Then it is left to the reader to check that, using these functions $F$ and $G$, the Gromov-Hausdorff distance between $(M, g(t))$ and $(\overline{M},d)$ tends to zero as $t \rightarrow T$.
\qed
\end{proof}

This completes the proof of Theorem \ref{thm3}.

\bigskip
\noindent
{\bf Acknowledgements.} \  The authors are grateful to Professor D.H. Phong for his advice, encouragement and support.  In addition, the first-named author thanks Professor G. Tian for some helpful discussions.  The  second-named author thanks Professor S. Casalaina-Martin for some useful conversations.  The authors are also grateful to Valentino Tosatti for some helpful comments on a previous draft of this paper.

\bigskip
\bigskip

$^{*}$ Department of Mathematics \\
Rutgers University, Piscataway, NJ 08854\\
Email: jiansong@math.rutgers.edu \\

$^{\dagger}$ Department of Mathematics \\
University of California San Diego, La Jolla, CA 92093\\
Email:  weinkove@math.ucsd.edu

\end{document}